\documentclass[11pt]{amsart}
\usepackage{amssymb,amsfonts,latexsym}
\usepackage[centertags]{amsmath}
\usepackage{amsfonts}
\usepackage{amssymb}
\usepackage{amsthm}
\usepackage[all]{xy}
\usepackage{cite}
\usepackage{graphicx,color}
\usepackage{todo}
\xyoption{poly}
\numberwithin{equation}{section} 

\font\tengothic=eufm10 scaled\magstep 1 \font\sevengothic=eufm7
scaled\magstep 1
\newfam\gothicfam
      \textfont\gothicfam=\tengothic
      \scriptfont\gothicfam=\sevengothic

\newtheorem{theorem}{Theorem}[section]
\newtheorem{lemma}[theorem]{Lemma}
\newtheorem{proposition}[theorem]{Proposition}

\newtheorem{conjecture}[theorem]{Conjecture}

\theoremstyle{definition}
\newtheorem{definition}[theorem]{Definition} 
\newtheorem{remark}[theorem]{Remark}

\newtheorem{question}[theorem]{Question}

\newcommand{\Soc}{\operatorname{Soc}}

\newcommand{\Ann}{\operatorname{Ann}}

\newcommand {\ZZ}{\mathbb{Z}}

\newcommand {\PP}{\mathbb{P}}

\newcommand{\fa}{\mathfrak a}

\begin{document}

\title[Harbourne, Schenck and Seceleanu's Conjecture]{Harbourne, Schenck and Seceleanu's Conjecture}

\author[R.\ M.\ Mir\'o-Roig]{Rosa M.\ Mir\'o-Roig${}^{*}$}
\address{Facultat de Matem\`atiques, Department d'\`Algebra i
 Geometria, Gran Via des les Corts Catalanes 585, 08007 Barcelona, Spain}
\email{miro@ub.edu}

\thanks{\noindent Printed \today \\
${}^{*}$  Part of the work for this paper was done while the
author was sponsored by MTM2013-45075-P.\\
}

\renewcommand{\subjclassname}{MSC2010}


\begin{abstract}

In \cite{HSS}, Conjecture 5.5.2, Harbourne, Schenck and Seceleanu conjectured that, for $r=6$ and all $r\ge 8$, the artinian ideal $I=(\ell _1^2,\dots ,l_{r+1}^2)\subset K[x_1, \dots ,x_r]$ generated by the square of $r+1$ general linear forms $\ell _{i}$ fails the Weak Lefschetz property. This paper is entirely devoted to prove this Conjecture. It is worthwhile to point out that  half of the  Conjecture -  namely, the case when the number of variables $r$ is even - was already proved in \cite{mmn}, Theorem 6.1. \end{abstract}

\maketitle

\tableofcontents

\section{Introduction}

Ideals generated by powers of linear forms have attracted great deal of attention recently. For instance, their Hilbert function  have been the focus of the papers \cite{AP}, \cite{sx} and \cite{HSS}; and the presence or failure of the Weak Lefschetz Property has been  deeply studied in \cite{HSS}, \cite{mmn} and \cite{ss}, among others.

Let $A = R/I$ be a standard graded artinian algebra, where $R = K[x_1,\dots,x_r]$ and $K$ is an algebraically closed field of characteristic 0.  If $\ell$ is a linear form then multiplication by $\ell$ induces a homomorphism from any graded component $[A]_i$ to the next and    the algebra $A$ is said to have the {\em Weak Lefschetz Property (WLP)} if the ``multiplication by a general linear form has maximal rank from each degree to the next." There has been a long series of papers determining classes of algebras holding/failing the WLP but much more work remains to be done.

The first result in this direction is due to Stanley \cite{stanley} and Watanabe \cite{watanabe} and it asserts that the WLP holds for an artinian complete intersection ideal generated by powers of linear forms. In \cite{ss}, Schenck and Seceleanu gave the nice result that in three variables {\em any} ideal generated by powers  of general linear forms has WLP. In contrast, in \cite{mmn}, we showed by examples that in 4 variables for $d=3, \dots , 12$ an ideal generated by the $d$-th power of five general  linear forms does not have the WLP. Therefore, it is natural to ask when the WLP holds for artinian ideals $I\subset K[x_1,\dots ,x_r]$ generated  by powers of $\ge r+1$ general linear forms. The goal of this short note is to solve Conjecture 5.5.2 in \cite{HSS} which says that for $r=6$ and $r\ge 8$, the artinian ideal $I=(\ell _1^2,\dots ,l_{r+1}^2)\subset K[x_1, \dots ,x_r]$ generated by the square of $r+1$ general linear forms $\ell _{i}$ fails the Weak Lefschetz property.  Half of the  Conjecture -  namely, the case when the number of variables $r$ is even - was proved in \cite{mmn}, Theorem 6.1 where first we use
 the inverse system dictionary to relate an ideal $I\subset K[x_1, \cdots ,x_r]$ generated by powers of linear forms to an ideal of fat points in $\PP^{r-1}$, and  then we show that the WLP problem of an ideal generated by powers of linear forms is  closely connected to the geometry  of the linear system  of hypersurfaces in $\PP^{r-2}$ of fixed degree  with preassigned multiple points.
 In this short note, we solve the remaining half - the case of an odd number of variables. The key point is to determine the degree where the WLP fails.

 \vspace{3mm}

Next, we outline the structure of the paper. In Section 2, we fix notation and we  briefly discuss general facts on Weak Lefschetz property needed later on. Section 3 is the heart of the paper and contains the proof of the conjecture stated by Harbourne, Schenck and Seceleanu in \cite{HSS}, Conjecture 5.5.2.


\section{Background and preparatory results}

In this section we fix notation, we recall the definition of Weak/Strong Lefschetz Property and we state some open conjectures related to the Weak Lefschetz property of ideals generated by powers of linear forms which are motivating  current research in this vast topic of research which touch numerous and different areas like algebraic geometry, commutative algebra and combinatorics.

Throughout this work  $K$ will be an algebraically closed field of characteristic zero.
 Given a graded artinian $K$-algebra $A=R/I$ where $R=K[x_1,x_2,\dots,x_r]$ and $I$ is a homogeneous ideal of $R$,
we denote by $H_A:\mathbb{Z} \longrightarrow \mathbb{Z}$ with $H_A(j)=\dim _K[A]_j$
its Hilbert function. Since $A$ is artinian, its Hilbert function is
captured in its {\em $h$-vector} $h=(h_0,h_1,\dots ,h_e)$ where $h_i=H_A(i)>0$ and $e$ is the last index with this property. The integer $e$ is called the {\em socle degree of} $A$.

\begin{definition}
Let $A=R/I$ be a graded artinian $K$-algebra. We say that $A$ is {\em level of Cohen-Macaulay type} $t$ if its socle is concentrated in one degree and has dimension $t$. E.g. a complete intersection is level of Cohen-Macaulay type 1 and a Gorenstein artinian $K$-algebra is also level of Cohen-Macaulay type 1.
\end{definition}

For  fixed
$e$ and $t$, a level graded artinian $K$-algebra of socle degree $e$, Cohen-Macaulay type $t$ and of maximal Hilbert function among
all level graded artinian $K$-algebras with that socle degree and Cohen-Macaulay type is said to be {\em
compressed}.  We can extend this notion as follows (see \cite{mmn2}, Definition 1.1, for more details):

\begin{definition} \label{def-compr}
Let $\fa \subset R$ be a homogeneous ideal.
    Then a level graded artinian $K$-algebra $A$ of socle degree $e$ and Cohen-Macaulay type $t$ is said to be
    {\it relatively compressed with respect to $\fa $} if $A$ has maximal length among all level graded artinian
    $K$-algebras $R/I$ satisfying
\begin{itemize}
\item[(i)] $\Soc R/I\cong  K(-e)^t$
\item[(ii)]
$\fa \subset I$.
\end{itemize}
\end{definition}

Equivalently, $A = R/I$ is relatively compressed with respect to $\fa $
if it is a quotient of $R/\fa $ having maximal length among all quotients of $R/\fa $ with
prescribed socle degree and Cohen-Macaulay type.

For level artinian algebras $A$ relatively compressed with respect to a complete
intersection $\fa \subset R$ we  have an upper bound for its Hilbert function. In fact, we have

\begin{lemma} \label{first bound}
Let $A =
R/I$ be a graded level artinian $K$-algebra of  socle  degree $e$, Cohen-Macaulay type $t$ and
relatively compressed with respect to a complete intersection $\fa
\subset R$.  Then, we have
\[
h_A(i) \leq \min \{ \dim [R/\fa]_i, t \cdot \dim [R/\fa ]_{e-i} \}.
\]
\end{lemma}

\begin{remark} \rm For $t=1$, we take $F \in
[\fa _e]^{\perp}$ and we
consider the Gorenstein artinian graded $K$-algebra $A=R/\Ann(F)$.
Because $R/\Ann(F)$ is a  quotient of $R/\fa $, which is Gorenstein, we
clearly have
\begin{equation} \label{hilb fn bound}
h_A(i) \leq \min \{ \dim [R/\fa ]_i, \dim [R/\fa ]_{e-i} \}.
\end{equation}
Note that by \cite{Iarrobino-LNM}, Theorem 4.16, if $\fa $ and $F$ are
both  general (or if $\fa $ is a monomial complete intersection and $F$
is general)  then we have equality in (\ref{hilb fn
bound}).
\end{remark}

\begin{definition}
Let $A=R/I$ be a graded artinian $K$-algebra. We say that $A$ has the {\em Weak Lefschetz Property} (WLP)
if there is a linear form $L \in [A]_1$ such that, for all
integers $i\ge0$, the multiplication map
\[
\times L: [A]_{i}  \longrightarrow  [A]_{i+1}
\]
has maximal rank, i.e.\ it is injective or surjective.
(We will often abuse notation and
say that the ideal $I$ has the WLP.) In this case, the linear form $L$ is called a Lefschetz
element of $A$. If for the general form $L \in [A]_1$ and for an integer number $j$ the
map $\times L:[A]_{j-1}  \longrightarrow  [A]_{j}$ does not have maximal rank, we will say that the ideal $I$ fails the WLP in
degree $j$.

$A$ has the {\em Strong Lefschetz Property} (SLP) if there is a linear form $L \in [A]_1$ such that, for all
integers $i\ge0$ and $k\ge 1$, the multiplication map
\[
\times L^k: [A]_{i}  \longrightarrow  [A]_{i+k}
\]
has maximal rank.
\end{definition}

In this paper we will deal with artinian ideals  $I=(\ell _1 ^{a_1},\dots ,\ell_{s}^{a_s})\subset K[x_1,\dots ,x_r]$ generated by powers of $s\ge r$ general linear forms $\ell _i$ and we relate them to the presence or failure of the WLP.
The first result in this topic is due to R. Stanley \cite{stanley} and J. Watanabe \cite{watanabe} and it has motivated this entire area of research. It says:

\begin{proposition}\label{WLPSLPci}
Let $R=K[x_1,\dots ,x_r]$ and let $\fa$ be an artinian monomial complete intersection, i.e. $\fa =(x_1^{a_1},\dots ,x_r^{a_r})$. Then, $R/I$ has both the WLP and the SLP.
\end{proposition}

It is worthwhile to point out that the above result is false in positive characteristic. Trying to extend the above proposition we can ask which ideals generated by powers of general linear forms define $K$-algebras that fail WLP. In two and three variables all such algebras satisfy the WLP. In fact, we have:

\begin{proposition}
  \label{cor:ss-result}
  (1) Any homogeneous artinian ideal in $K[x,y]$ has WLP. In particular, every artinian ideal generated by powers of linear forms in 2 variables has the WLP.

(2) Every artinian ideal generated by powers of general linear forms in 3 variables has the WLP.
\end{proposition}
\begin{proof} (1) See \cite{HMNW}, Proposition 4.4.

(2) See \cite{ss}, Theorem 2.4.
\end{proof}

The above proposition is no longer true for ideals generated by powers of linear forms in $r\ge 4$ variables. In fact, in \cite{mmn} we prove that $A:=K[x_1,x_2,x_3,x_4]/ (x_1^3,x_2^3,x_3^3,x_4^4, (x_1+x_2+x_3+x_4)^3)$ does not have the WLP. The $h$-vector  of $A$ is $(1,4,10,15,15,6)$  and $A_3\longrightarrow A_4$ has not maximal rank. The following  questions naturally arise from this example:

\begin{question} (1) (Almost complete intersections and uniform powers) Let $I=(\ell _1^t,\dots ,\ell_{r+1}^t)\subset K[x_1,\dots ,x_r]$ be an almost complete intersection ideal generated by uniform powers of general linear forms. For which values of $r$ and $t$ does $K[x_1,\dots ,x_r]/I$ fail WLP?

(2) (Uniform powers) Let $I=(\ell _1^t,\dots ,\ell_{s}^t)\subset K[x_1,\dots ,x_r]$ be an artinian ideal generated by uniform powers of general linear forms. For which values of $r$, $s$ and $t$ does $K[x_1,\dots ,x_r]/I$ fail WLP?

(3) (Mixed powers) Let $I=(\ell _1^{a_1},\dots ,\ell_{s}^{a_s})\subset K[x_1,\dots ,x_r]$ be an artinian ideal generated by  powers of general linear forms. For which values of $r$, $s$ and $a_i$ does $K[x_1,\dots ,x_r]/I$ fail WLP?

\end{question}

Nice contributions to these problems are given in \cite{ss}, \cite{mmn} and \cite{HSS} but no complete answers to them  are known and more work has to be done. Here there are  three open conjectures from \cite{mmn} and\cite{HSS}  related to these questions:

\begin{conjecture} \label{conj1} (\cite{HSS}, Conjecture 5.5.2) For $r=6$ and all $r\ge 8$, $K[x_1,\dots ,x_r]/(\ell _1^2,\dots ,\ell_{r+1}^2)$ where $\ell _i$ are general linear forms, fails WLP.
\end{conjecture}

\begin{conjecture} (\cite{HSS}, Conjecture 1.2) For $I=(\ell _1^t,\dots ,\ell_n^t)\subset K[x_1,\dots ,x_r]$ with $\ell _i$ general linear forms and $n\ge r+1\ge 5$, $A=K[x_1,\dots ,x_r]/I$ fails WLP for $t\gg 0$.
\end{conjecture}

\begin{conjecture} (\cite{mmn}, Conjecture 6.6) Let $R = K[x_1,\dots,x_{2n+1}]$, where $n \geq 4$.
Let $\ell \in R$ be a general linear form, and let $I = \langle x_1^{d},
\cdots ,x_{2n+1}^d,\ell ^{d} \rangle$.
Then the ring $R/I$ fails the WLP if and only if $d>1$.  Furthermore, if $n=3$ then $R/I$ fails the WLP when $d=3$.
\end{conjecture}

This paper will be devoted to proving Conjecture \ref{conj1}. More precisely, we will prove that
for almost complete intersection ideals  generated by squares of general linear forms we have the following full characterization which solves Conjecture \ref{conj1}.

\begin{theorem}\label{main} Set $A=K[x_1,\dots ,x_r]/(\ell_1^2,\dots ,\ell _{r+1}^2)$ where $\ell _i$ are general linear forms.
We have:
\begin{itemize}
\item[(1)] If $2\le r \le 5$ or $r=7$ then $A$  has the WLP.
\item[2)] If $r=6$ or $r\ge 8$ then $A$ fails WLP.
\end{itemize}
\end{theorem}


\section{Proof of the main Theorem}

This section is entirely devoted to prove Theorem \ref{main}. In order to make the proof self-contained we start with a series of technical lemmas and propositions which have as a goal to compute the Hilbert function of an almost complete intersection ideal generated by the quadratic powers of general linear forms.

\begin{lemma}\label{HF_ci} Let $\fa =(x_1^2,x_2^2,\dots ,x_r^2)$. Then the Hilbert function of $A=K[x_1,\dots ,x_r]/\fa $ is
$$H_A(j)=\begin{cases} {r\choose j} & \text{ if } 0\le j \le r, \\ 0 & \text{ otherwise. }\end{cases}$$
\end{lemma}

\begin{proof} To compute the Hilbert function $H_A(t)$ of $A$, it is enough to describe bases $B_j$ of $[A]_j$. We choose the residue classes of the elements in the following sets
$$B_j=\{ x_{i_1}\cdot \dots \cdot x_{i_j} \mid 1\le i_1< \dots < i_j \le r \} \\ \text{ for } \\  1\le j \le r, $$
and this concludes the proof.
\end{proof}

Given the polynomial ring $R=K[x_1, \dots ,x_r]$ we will consider its Macaulay-Matlis inverse system $E=K[X_1,\dots ,X_r]$ which means that $E$ is a graded $R$-module and $x_i$ acts as $\frac{\partial }{\partial X_i}$.

\begin{proposition} \label{propietats_gor} Set $g:=\sum _{1\le i_1<\cdots <i_{r-2}\le r}X_{i_1}X_{i_2}\cdot \dots \cdot X_{i_{r-2}}$ and let $G:=K[x_1, \dots , x_r]/\Ann(g)$ be the associated artinian Gorenstein $K$-algebra. The following hold:
\begin{itemize}
\item[(1)] The Hilbert function of $G$ is given by
$$H_{G}(t)=\begin{cases} 0  & \text{ if } t<0 \text{ or } t>r-2, \\
{r\choose t} & \text{if } 0\le t \le \frac{r-2}{2}, \\
{r\choose r-2-t} & \text{ if } \frac{r-2}{2}<t\le r-2.
\end{cases}$$
\item[(2)] $G$ has socle degree $r-2$ and it is relatively compressed with respect to the  complete intersection monomial ideal $\fa=(x_1^2,x_2^2,\dots ,x_r^2)$.

    \item[(3)] For $r\ge 3$ odd, we define
     $$S:=\{x_1^2,x_2^2,\cdots ,x_r^2, (x_{i_1}-x_{i_2}) (x_{i_3}-x_{i_4})\cdot \dots \cdot (x_{i_{r-2}}-x_{i_{r-1}}) \}$$  where $1\le i_j\le r$ and $i_j\ne i_{s}.$ We have $\langle S \rangle \subset \Ann(g)$.
        \item[(4)] For $r\ge 4$ even, we define   $$S:=\{x_1^2,x_2^2,\cdots ,x_r^2, (x_{i_1}-x_{i_2}) (x_{i_3}-x_{i_4})\cdot \dots \cdot (x_{i_{r-3}}-x_{i_{r-2}})x_{i_{r-1}}´\} $$  where $1\le i_j\le r$ and $i_j\ne i_{s}.$ We have $\langle S \rangle \subset \Ann(g)$.
\end{itemize}
\end{proposition}

\begin{proof} (1) Since $G$ is an artinian Gorenstein $K$-algebra, the Hilbert function of $G$ is symmetric, i.e. $H_G(i)=H_G(r-2-i)$ for $0\le i\le r-2$ and we only have to compute $H_G(j)$ for $ \frac{r-2}{2}\le j \le r-2$. By Macaulay-Matlis duality for any $\frac{r-2}{2}\le j \le r-2$, we have:
$$\begin{array}{rcl} H_G(j) & = &
\dim _K\langle \frac{\partial ^{j} g}{\partial X_{1}^{a_1} \cdots \partial X_{r}^{a_r} } \mid 0\le a_i \text{ and } a_1+\dots +a_r=j    \rangle  \\
& = & \dim _K\langle \frac{\partial ^{j} g}{\partial X_{i_1}\cdots \partial X_{i_j} } \mid 0\le i_1< \dots <i_j\le r    \rangle \\
& = & {r\choose j}.\end{array}$$

\vskip 2mm
(2) By definition $A$ has socle degree $r-2$. Using (1) and Lemma \ref{HF_ci} we check that
$$
H_G(i) = \min \{ \dim [R/\fa ]_i, \dim [R/\fa ]_{r-2-i} \}
$$
and we deduce from Lemma \ref{first bound} that $G$ is relatively compressed with respect to the  complete intersection monomial ideal $\fa$.

\vskip 2mm
(3) Set $C=K[x_1,\dots ,x_r]/\fa$. By item (2), $H_G(j)=H_C(j)$ for all $0\le j \le \frac{r-2}{2}$ and we clearly have $\fa \subset \Ann(g)$ since $x_i^2 \circ g=\frac{\partial ^2g}{\partial X_i^2}=0$ for all $0\le i \le r$. Therefore, for $r\ge 7$,  the only generators of $\Ann(g)$ of degree $\le \frac{r-2}{2}$ are: $x_1^2,x_2^2,\cdots ,x_r^2$.

\vskip 2mm
\noindent {\bf Claim:} $(x_{i_1}-x_{i_2}) (x_{i_3}-x_{i_4})\cdot \dots \cdot (x_{i_{r-2}}-x_{i_{r-1}})\in \Ann(g)$.

\noindent {\bf Proof of the Claim.} Obviously, it is enough to check one case, namely $(x_{1}-x_{2}) (x_{3}-x_{4})\cdot \dots \cdot (x_{r-2}-x_{r-1})\circ g=0$.  We set $r=2q+1$ and we proceed by induction on $q$.

For $q=1$ and $2$ the claim is true since  we have $(x_1-x_2)\circ (X_1+X_2+X_3)$=0 and  $(x_1-x_2)(x_3-x_4)\circ (X_1X_2X_3+X_1X_2X_4+X_1X_2X_5+X_1X_3X_4+X_1X_3X_5+X_1X_4X_5+X_2X_3X_4+X_2X_3X_5+X_2X_4X_5+X_3X_4X_5)=0$.

Assume $q\ge 3$. We define
$$\begin{array}{rcl} g_1 & := & \sum _{1\le j_1<\dots <j_{r-3}\le r \atop j_i\ne r-2} X_{j_1}\cdot \dots \cdot X_{j_{r-3}} \\ & = & X_r\sum _{1\le j_1<\dots <j_{r-4}\le r-1 \atop j_i\ne r-2} X_{j_1}\cdot \dots \cdot X_{j_{r-4}}+\sum _{1\le j_1<\dots <j_{r-3}\le r-1 \atop j_i\ne r-2} X_{j_1}\cdot \dots \cdot X_{j_{r-3}} \\ & =: & X_rg_1^1+g_1^2\end{array} $$
and, analogously,
$$\begin{array}{rcl} g_2 & := & \sum _{1\le j_1<\dots <j_{r-3}\le r \atop j_i\ne r-1} X_{j_1}\cdot \dots \cdot X_{j_{r-3}} \\ & = &  X_r\sum _{1\le j_1<\dots <j_{r-4}\le r-1 \atop j_i\ne r-1} X_{j_1}\cdot \dots \cdot X_{j_{r-4}}+\sum _{1\le j_1<\dots <j_{r-3}\le r-1 \atop j_i\ne r-1} X_{j_1}\cdot \dots \cdot X_{j_{r-3}}\\ & =: & X_rg_2^1+g_2^2.\end{array} $$

We have:
$$\begin{array}{rcl} (x_{1}-x_{2}) (x_{3}-x_{4})\cdot \dots \cdot (x_{r-2}-x_{r-1})\circ g & = & \\
(x_{1}-x_{2}) \cdot \dots \cdot (x_{r-4}-x_{r-3})x_{r-2} \circ g - (x_{1}-x_{2}) \cdot \dots \cdot (x_{r-4}-x_{r-3})x_{r-1} \circ g & = & \\
(x_{1}-x_{2}) \cdot \dots \cdot (x_{r-4}-x_{r-3}) \circ g_1- (x_{1}-x_{2}) \cdot \dots \cdot (x_{r-4}-x_{r-3}) \circ g_2 & = & \\
(x_{1}-x_{2}) \cdot \dots \cdot (x_{r-4}-x_{r-3}) \circ (X_rg_1^1+g_1^2)- (x_{1}-x_{2}) \cdot \dots \cdot (x_{r-4}-x_{r-3}) \circ (X_rg_2^1+g_2^2).
\end{array}$$

Using the induction hypothesis  we get
$$(x_{1}-x_{2}) \cdot \dots \cdot (x_{r-4}-x_{r-3}) \circ X_rg_1^1 = 0 $$
and
$$(x_{1}-x_{2}) \cdot \dots \cdot (x_{r-4}-x_{r-3}) \circ X_rg_2^1= 0.$$

Therefore, we have
$$\begin{array}{rcl} (x_{1}-x_{2}) (x_{3}-x_{4})\cdot \dots \cdot (x_{r-2}-x_{r-1})\circ g & = & \\
(x_{1}-x_{2}) \cdot \dots \cdot (x_{r-4}-x_{r-3}) \circ g_1^2- (x_{1}-x_{2}) \cdot \dots \cdot (x_{r-4}-x_{r-3}) \circ g_2^2  & = & \\
(x_{1}-x_{2}) \cdot \dots \cdot (x_{r-4}-x_{r-3}) \circ (X_1\cdot \dots \cdot X_{r-3})+X_{r-1}h) - & & \\
(x_{1}-x_{2}) \cdot \dots \cdot (x_{r-4}-x_{r-3}) \circ (X_1\cdot \dots \cdot X_{r-3})+X_{r-2}h)
\end{array}$$

where $$h=\sum _{1\le s_1<\dots <s_{r-4}\le r-3} X_{s_1}\cdot \dots \cdot X_{s_{r-4}}$$
Applying again hypothesis of induction we obtain
$$(x_{1}-x_{2}) \cdot \dots \cdot (x_{r-4}-x_{r-3}) \circ X_{r-1}h= 0 $$
$$(x_{1}-x_{2}) \cdot \dots \cdot (x_{r-4}-x_{r-3}) \circ X_{r-2}h= 0$$

\noindent and, hence, we  get

 $$\begin{array}{rcl} (x_{1}-x_{2}) (x_{3}-x_{4})\cdot \dots \cdot (x_{r-2}-x_{r-1})\circ g & = & \\
(x_{1}-x_{2}) \cdot \dots \cdot (x_{r-4}-x_{r-3}) \circ (X_1\cdot \dots \cdot X_{r-3})+X_{r-1}h) & & \\ - (x_{1}-x_{2}) \cdot \dots \cdot (x_{r-4}-x_{r-3}) \circ (X_1\cdot \dots \cdot X_{r-3})+X_{r-2}h) & = & \\
(x_{1}-x_{2}) \cdot \dots \cdot (x_{r-4}-x_{r-3}) \circ X_1\cdot \dots \cdot X_{r-3} - & & \\(x_{1}-x_{2}) \cdot \dots \cdot (x_{r-4}-x_{r-3}) \circ X_1\cdot \dots \cdot X_{r-3} & = & \\
0.
\end{array}$$

\vskip 2mm
(4) It is analogous to (3) and we leave it to the reader.

\end{proof}

\begin{remark} \rm In this paper we do not need an explicit description of a (minimal) system of generators of $\Ann(g)$. Computer evidences, using Macaulay2, suggest that $S$ is a full system of generators of $\Ann(g)$ and, hence $\Ann(g)$ is generated by $r$ quadrics and ${r-1\choose q-1}$ forms of degree $q$ if $r=2q+1\ge 7$ and by $r$ quadrics and ${r\choose q}-{r\choose q-2}$ forms of degree $q$ if $r=2q\ge 6$.

\end{remark}

\begin{proposition} \label{acm}   Fix an integer $r\ge 3$. Set $g:=\sum _{1\le i_1<\cdots <i_{r-2}\le r}X_{i_1}X_{i_2}\cdot \dots \cdot X_{i_{r-2}}$. Let $G:=K[x_1, \dots , x_r]/\Ann(g)$ be the associated artinian Gorenstein $K$-algebra relatively compressed with respect to the monomial complete intersection  $\fa =(x_1^2,\dots ,x_r^2)$. We define $J=(\fa : \Ann(g))$. We have:
\begin{itemize}
\item[(1)] $J$ is an almost complete intersection artinian ideal generated by $r+1$ forms of degree 2 and with socle degree $$e(J)=\begin{cases}  q+1 & \text{ if } r=2q+1, \\ q  & \text{ if } r=2q.
    \end{cases}$$
\item[(2)] If $r=2q+1$, the Hilbert function of $A:=K[x_1,\dots, x_r]/J$ is given by
$$H_A(t)=\begin{cases} 0 & \text{ if } t<0 \text{ or } t> q+1, \\
  {r\choose t}-{r\choose t-2} & \text{ if } 0\le t \le q, \\
   {r\choose q}-{r\choose q-1} & \text{ if }  t=q+1. \end{cases}$$
   If $r=2q$, we have
   $$H_A(t)=\begin{cases} 0 & \text{ if } t<0 \text{ or } t> q,\\
  {r\choose t}-{r\choose t-2} & \text{ if } 0\le t \le q.
   \end{cases}$$
\item[(3)] $J=(x_1^2,x_2^2,\dots ,x_r^2,(x_1+\dots +x_r)^2)$.
\end{itemize}
\end{proposition}

\begin{proof} (1) Since $G$ is an artinian Gorenstein $K$-algebra relatively compressed with respect to $\fa $ and with socle degree $r-2$ and $J$ is linked to $\Ann (g)$ by means of $\fa $, we  have that $J$ is an almost complete intersection generated by $x_1^2, \dots , x_r^2, f$ where $f$ is a form of degree 2.

According to the proof of Proposition \ref{propietats_gor} (3) and (4), $\Ann(g)$ has $r$ minimal generators of degree 2, at least one minimal generator of degree  $\lfloor \frac{r}{2} \rfloor $ and all the others  (if any) in degree $\ge \lfloor \frac{r}{2} \rfloor $. Therefore, the socle degree of $J$ is
$$e(J)=-\lfloor \frac{r}{2} \rfloor +r=\begin{cases}  q+1 & \text{ if } r=2q+1, \\ q & \text{ if } r=2q
    \end{cases}$$
which proves what we want.

\vskip 2mm
(2) Since the Hilbert function of $K[x_1, \dots , x_r]/\fa$ is known (Lemma \ref{HF_ci}) and of $K[x_1, \dots , x_r]/\Ann(g)$ is also known (Proposition \ref{propietats_gor}
), we can compute the Hilbert function of $J=(\fa :\Ann(g))$ and we have
$$ H_{K[x_1, \dots , x_r]/J}(t)=H_{K[x_1, \dots , x_r]/\fa}(t)-H_G(r-t)$$
and we get what we want.

\vskip 2mm
(3) By the proof of (1) we know that $J=(x_1^2,x_2^2,\dots ,x_r^2,f)$ where $f$ is an homogeneous form of degree 2. Let us determine it. More precisely, let us chech that we can take $f=(x_1+\dots +x_r)^2$. We assume that $r$ is odd, say, $r=2q+1$. Analogous argument works for $r$ even. Set
 $$S:=\{x_1^2,x_2^2,\cdots ,x_r^2, (x_{i_1}-x_{i_2}) (x_{i_3}-x_{i_4})\cdot \dots \cdot (x_{i_{r-2}}-x_{i_{r-1}}) \}$$  where $1\le i_j\le r$ and $i_j\ne i_{s}.$ By Proposition \ref{propietats_gor}, $S$ is part of a system (not necessarily minimal) of generators of $\Ann(g)$. Therefore, it is enough to prove:
 \begin{itemize}
 \item[i)] For all $h\in S$, $h\cdot (x_1+\dots +x_r)^2\in \fa$.
 \item[(ii)] If $f\in K[x_1,\dots ,x_r]_2$ and $h\cdot f\in \fa$ for all $h\in S$ then $f\in (x_1^2,x_2^2,\dots ,x_r^2,(x_1+\dots +x_r)^2)$.
 \end{itemize}

 Let us prove (i). By symmetry we only have to check that $(x_{1}-x_{2}) (x_{3}-x_{4})\cdot \dots \cdot (x_{r-2}-x_{r-1})\cdot (x_1+\dots +x_r)^2\in \fa$ or, equivalently, $(x_{1}-x_{2}) (x_{3}-x_{4})\cdot \dots \cdot (x_{r-2}-x_{r-1})\cdot (\sum _{1\le i_1<i_2\le r } x_{i_1}x_{i_2})\in \fa$. We proceed by induction on $q$. For $q=1$ we have $(x_1-x_2)(x_1x_2+x_1x_3+x_2x_3)=x_1^2(x_2+x_3)-x_2^2(x_1+x_3)\in \fa$. Assume $q\ge 2$. By hypothesis of induction, it holds $$(x_{3}-x_{4})\cdot \dots \cdot (x_{r-2}-x_{r-1})\cdot (\sum _{3\le i_1<i_2\le r } x_{i_1}x_{i_2})\in (x_3^2,\dots ,x_r^2).$$
 Therefore, we have
 $$\begin{array}{rcl}(x_{1}-x_{2}) (x_{3}-x_{4})\cdot \dots \cdot (x_{r-2}-x_{r-1})\cdot (\sum _{1\le i_1<i_2\le r } x_{i_1}x_{i_2}) & = & (\text{ mod. } \fa) \\
 (x_{1}-x_2) (x_{3}-x_{4})\cdot \dots \cdot (x_{r-2}-x_{r-1})\cdot (x_1(x_2+\dots +x_r)+x_2(x_3+\dots x_r)) & = & (\text{ mod. } \fa) \\
 x_{1} (x_{3}-x_{4})\cdot \dots \cdot (x_{r-2}-x_{r-1})\cdot (x_2(x_3+\dots x_r)) & & \\
 -x_2 (x_{3}-x_{4})\cdot \dots \cdot (x_{r-2}-x_{r-1})\cdot (x_1(x_2+\dots +x_r))
  & = & (\text{ mod. } \fa) \\
  0 & &  (\text{ mod. } \fa) .
 \end{array}$$

 To prove (ii) we will check that if  $f\notin (x_1^2,\dots ,x_r^2,(x_1+\dots +x_r)^2)= (x_1^2,\dots ,x_r^2,\sum _{1\le i_1<i_2\le r}x_{i_1}x_{i_2})$ then $\exists h\in S$ such that $f\cdot h\notin \fa$. Reordering the variables, if necessary, any form of degree two $f\notin (x_1^2,\dots ,x_r^2,\sum _{1\le i_1<i_2\le r}x_{i_1}x_{i_2})$ can be written as
 $$f=a_1x_1^2+\dots +a_rx_r^2+a(\sum _{1\le i_1<i_2\le r}x_{i_1}x_{i_2})+\sum_{1\le i_1<i_2\le r} b_{i_1i_2}x_{i_1}x_{i_2}$$
 with $b_{12}=0$ and $b_{13}
 \neq 0$ or $b_{12}=\dots =b_{1r}=0$ and $b_{23}\neq 0$.
 If $b_{12}=0$ and $b_{13}
 \neq 0$, then $$(x_{2}-x_{3}) (x_{4}-x_{5})\cdot \dots \cdot (x_{r-1}-x_{r})\cdot (\sum _{1\le i_1<i_2\le r} b_{i_1i_2} x_{i_1}x_{i_2}) = b_{13}x_1x_2x_3\prod _{j=2}^qx_{2j} + (\text{ other monomials )}\notin \fa ,$$
 and, if $b_{12}=\dots =b_{1r}=0$ and $b_{23}\neq 0$, then $$(x_{1}-x_{2}) (x_{3}-x_{4})\cdot \dots \cdot (x_{r-2}-x_{r-1})\cdot (\sum _{1\le i_1<i_2\le r} b_{i_1i_2} x_{i_1}x_{i_2}) = b_{13}x_1x_2x_3\prod _{j=2}^qx_{2j} + (\text{ other monomials )}\notin \fa .$$
\end{proof}

\begin{lemma}\label{35}  Set $A=K[x_1,\dots ,x_r]/(x_1^2,\dots ,x_r^2,\ell ^2)$, $B=K[x_1,\dots ,x_r]/(\ell _1^2,\dots ,\ell _r^2,\ell _{r+1}^2)$ and $C=K[x_1,\dots ,x_r]/(q_1,\dots ,q_r,q_{r+1})$ where $\ell $ and $\ell _{i}$ are general linear forms and $q_j$ are general forms of degree 2. Then
$$H_A(t)=H_B(t)=H_C(t) \text{ for all } t\in \ZZ.$$
\end{lemma}
\begin{proof} By semicontinuity we have
\begin{equation} \label{keyinequality} H_A(t)\le H_B(t)\le H_C(t) \text{ for all } t\in \ZZ.
\end{equation}
By Proposition \ref{acm}(2),  if  $r=2q+1$ we have
$$H_A(t)=\begin{cases} 0 & \text{ if } t<0 \text{ or } t> q+1, \\
  {r\choose t}-{r\choose t-2} & \text{ if } 0\le t \le q, \\
   {r\choose q}-{r\choose q-1} & \text{ if }  t=q+1 \end{cases}$$
   and, if $r=2q$, we have
   $$H_A(t)=\begin{cases} 0 & \text{ if } t<0 \text{ or } t> q\\
  {r\choose t}-{r\choose t-2} & \text{ if } 0\le t \le q.
   \end{cases}$$
By Proposition \ref{WLPSLPci}, $J=(q_1,\dots ,q_r)$ has the SLP. Hence, we have
$$H_C(t)=max\{ H_{K[x_1,\dots ,x_r]/J}(t)- H_{K[x_1,\dots ,x_r]/J}(t-2),0 \}=H_A(t)$$
 and we conclude $H_C(t)=H_A(t)$ for all $t\in \ZZ$. Therefore, using the inequalities (\ref{keyinequality}), we get $H_A(t)=H_B(t)=H_C(t)$ for all $t\in \ZZ$ which proves what we want.
\end{proof}

We are now ready to prove the main result of this paper.

\noindent {\em Proof of Theorem} \ref{main}
(1) For $2\le r \le 5$ or $r=7$ we can use Macaulay2 to check that $A$ has the WLP.

(2) For $r$ even the reader can see \cite{mmn}, Theorem 6.1. Let us assume $r=2q+1$, $q\ge 4$. Set $I=(\ell _1^2,\dots , \ell _{r+1}^2)\subset R:=K[x_1,\dots ,x_r]$ where $\ell _i$ are general linear forms. We  consider the exact sequence :
$$ 0\longrightarrow [I:\ell]/I \longrightarrow R/I \stackrel{\times \ell}{\longrightarrow} R/I(1) \longrightarrow R/(I,\ell )(1)\longrightarrow 0$$
where $\ell $ is a general linear form.
The multiplication by $\ell$  will fail to have maximal rank from degree $i-1$ to degree $i$ exactly when
\begin{equation}
  \label{eq:max-rank-I}
\dim_K [R/(I,\ell)]_i \neq \max\{ \dim_K [R/I]_i - \dim_K [R/I]_{i-1} , 0 \}.
\end{equation}

By Lemma \ref{35} and Proposition \ref{acm} we have
$$\dim _K(R/I)_q={2q+1\choose q}-{2q+1\choose q-2} \text{ and } \dim _K(R/I)_{q-1}={2q+1\choose q-1}-{2q+1\choose q-3}.$$
Since $\dim _K[R/I]_q> \dim _K[R/I]_{q-1}$, to prove that $R/I$ fails WLP it will be enough to check that
 $$\dim _K[R/I]_q<\dim_K[R/(I,\ell)]_q+ \dim _K[R/I]_{q-1}.$$
 By \cite{sx}, Corollaries 7.3 and 7.4, we have $\dim_K[R/(I,\ell)]_q=2^q$.  Therefore, we have to prove $$\dim _K[R/I]_q<2^q+ \dim _K[R/I]_{q-1}.$$ Let us prove it. Since ${n-1\choose k}-{n-1\choose k-1}=\frac{n-2k}{n}{n\choose k}$ and ${n\choose k}=\frac{n+1-k}{k}{n\choose k-1}$, we obtain
$$\dim _K[(R/I)]_q={2q+1\choose q}-{2q+1\choose q-2}<2^q+ \dim _K[(R/I)]_{q-1}=2^q+{2q+1\choose q-1}-{2q+1\choose q-3}$$ if and only if
$${2q+1\choose q}-{2q+1\choose q-1}<2^q+ {2q+1\choose q-2}-{2q+1\choose q-3}$$
if and only if
$${2q+2\choose q}\frac{1}{q+1}<2^q+ {2q+2\choose q-2}\frac{3}{q+1}$$
if and only if
$$2^q>\frac{1}{q+1}[{2q+2\choose q}-3{2q+2\choose q-2}=\frac{1}{q+1}{2q+2\choose q}[1-3\frac{(q-1)q}{(q+4)(q+3)}].$$
We easily check that the last equality is true for $q=4$ or $5$ and it is obviously true for $q\ge 6$ since we have $$1-3\frac{(q-1)q}{(q+4)(q+3)}=\frac{-2q^2+10q+12}{(q+4)(q+3)}=\frac{-2(q-6)(q+1)}{(q+4)(q+3)}\le 0.$$
\qed

It is worthwhile to point out that if instead of considering an almost complete intersection ideal $I=(\ell _1^2,\dots , \ell _{8}^2)\subset R:=K[x_1,\dots ,x_7]$ generated by the square of general linear forms $\ell _i$, we take an almost complete intersection ideal  $J=(q_1, \dots , q_{8})$ generated by general quadrics then  $J$ {\em do} have the WLP. We observe that $I$ and $J$ have the same Hilbert function  but the WLP behaviour is very different. This example illustrates how subtle the WLP is since  a minuscule change in the ideal  can affect the WLP.



\begin{thebibliography}{999}

\bibitem{AP} F.\ Ardila and A.\ Postnikov, {\em Combinatorics and geometry of power ideals}, Trans.\ Amer.\ Math.\ Soc.\ {\bf 362} (2010), 4357--4384.




\bibitem{HSS} B.\ Harbourne, H.\ Schenck and A.\ Seceleanu, {\em Inverse systems, Gelfand-Tsetlin patterns and the Weak Lefschetz Property},  J. Lond. Math. Soc. {\bf 84} (2011), no. 3, 712–-730.

\bibitem{HMNW} T.\ Harima, J.\ Migliore, U.\ Nagel, and J.\
Watanabe,  {\em The Weak and Strong     Lefschetz Properties for
artinian $K$-Algebras}, J.\ Algebra {\bf 262} (2003), 99--126.


\bibitem{Iarrobino-LNM}
A.\ Iarrobino, V.\ Kanev, ``Power sums, Gorenstein algebras, and
determinantal loci,'' Lecture Notes in Mathematics {\bf 1721},
Springer-Verlag, 1999.


\bibitem{mmn} J.\ Migliore, R.M. \ Mir\'o-Roig and U.\ Nagel, {\em On the Weak Lefschetz Property for Powers of Linear Forms},  Alg. and Number Theory    {\bf 6} (2012), 488--526.

\bibitem{mmn2} J.\ Migliore, R.M.\ Mir\'o-Roig and U.\ Nagel, {\em  Minimal resolution of Relatively Compressed Level Algebras},
 J. Alg. {\bf 284} (2005),  333--370.



\bibitem{ss} H.\ Schenck and A.\ Seceleanu, {\em The Weak Lefschetz Property and powers of linear forms in $K[x,y,z]$}, Proc.\ Amer.\ Math.\ Soc.\ {\bf 138} (2010), 2335--2339.

\bibitem{stanley} R.\ Stanley, {\em Weyl groups, the hard Lefschetz
theorem, and the Sperner property}, SIAM J.\ Algebraic Discrete Methods
{\bf 1} (1980), 168--184.

\bibitem{sx} B. \ Sturmfels and Z. \ Xu, {\em Sagbi bases of Cox-Nagata
rings}, J. Eur. Math. Soc. {\bf 12} (2010),  429--459.

\bibitem{watanabe} J.\ Watanabe, {\em The Dilworth number of Artinian
rings and finite posets with rank function}, Commutative Algebra and
Combinatorics, Advanced Studies in Pure Math.\ Vol.\ 11, Kinokuniya Co.\
North Holland, Amsterdam (1987), 303--312.

\end{thebibliography}
\end{document}